\documentclass[11pt, draft]{amsart}
\usepackage{amssymb, amstext, amscd, amsmath, paralist}
\usepackage{xypic, color}
\usepackage[shortlabels]{enumitem}

\makeatletter
\def\@cite#1#2{{\m@th\upshape\bfseries%
[{#1\if@tempswa{\m@th\upshape\mdseries, #2}\fi}]}}
\makeatother

\theoremstyle{plain}
\newtheorem{theorem}{Theorem}[section]
\newtheorem{corollary}[theorem]{Corollary}
\newtheorem{proposition}[theorem]{Proposition}
\newtheorem{lemma}[theorem]{Lemma}

\theoremstyle{definition}
\newtheorem{definition}[theorem]{Definition}
\newtheorem{example}[theorem]{Example}
\newtheorem{examples}[theorem]{Examples}

\newtheorem{remark}[theorem]{Remark}

\theoremstyle{remark}

\def\bbC{\mathbb C}

\def\bbF{\mathbb F}
\def\bbN{\mathbb N}

\def\bbT{\mathbb T}

     \newcommand{\sB}{\mathcal B}
\newcommand{\fC}{\mathfrak C}     
     
     \newcommand{\sE}{\mathcal E}
     \newcommand{\sF}{\mathcal F}
     
     \newcommand{\sH}{\mathcal H}
     \newcommand{\sI}{\mathcal I}

     \newcommand{\sL}{\mathcal L}
     \newcommand{\sM}{\mathcal M}
     \newcommand{\sN}{\mathcal N}
     \newcommand{\sO}{\mathcal O}

     \newcommand{\sT}{\mathcal T}

     \newcommand{\sX}{\mathcal X}

\newcommand{\al}{\alpha}
\newcommand{\be}{\beta}

\newcommand{\vpi}{\varphi}
\newcommand{\si}{\sigma}
\newcommand{\la}{\lambda}
\newcommand{\de}{\delta}

\newcommand{\ca}{\mathrm{C}^*}

\newcommand{\foral}{\text{ for all }}
\newcommand{\id}{{\operatorname{id}}}

\newcommand{\inte}{\operatorname{int}}

\newcommand{\nor}[1]{\left\Vert #1\right\Vert}     
\newcommand{\sca}[1]{\left\langle#1\right\rangle}  

\addtocontents{toc}{\protect\setcounter{tocdepth}{1}}

\begin{document}

\title[{Ergodic extensions of endomorphisms of MASAS}]{Ergodic extensions and Hilbert modules associated to endomorphisms of MASAS}

\author[E.T.A. Kakariadis]{Evgenios T.A. Kakariadis}
\address{School of Mathematics and Statistics\\ Newcastle University\\ Newcastle upon Tyne\\ NE1 7RU\\ UK}
\email{evgenios.kakariadis@ncl.ac.uk}

\author[J.R. Peters]{Justin R. Peters}
\address{Department of Mathematics\\Iowa State University\\ Ames\\ Iowa\\ IA 50011\\ USA}
\email{peters@iastate.edu}

\thanks{The second author acknowledges partial support from the National Science Foundation, DMS-0750986}
\maketitle

\begin{abstract}
We show that a class of ergodic transformations on a probability measure space $(X,\mu)$ extends to a representation of $\mathcal{B}(L^2(X,\mu))$ that is both implemented by a Cuntz family and ergodic. This class contains several known examples, which are unified in this work.

During the analysis of the existence and uniqueness of such a Cuntz family we give several results of individual interest. Most notably we prove a decomposition of $X$ for $N$-to-one local homeomorphisms that is connected to the orthonormal basis of Hilbert modules. We remark that the trivial Hilbert module of the Cuntz algebra $\mathcal{O}_N$ does not have a well-defined Hilbert module basis (moreover that it is unitarily equivalent to the module sum $\sum_{i=1}^n \mathcal{O}_N$ for infinitely many $n \in \mathbb{N}$).
\end{abstract}

\section{Introduction}

In this paper we continue with the examination of representations of dynamical systems implemented by Cuntz families \cite{KakPet13}. We are strongly motivated by the recent work and aspect of Courtney, Muhly and Schmidt \cite{CouMuhSch10} in which the general theory of Hilbert modules is used as an alternate route to examine specific examples, in association with earlier work of Laca \cite{Lac93} where ergodic transformations of $\sB(H)$ in general are examined. In contrast to our previous work \cite{KakPet13}, which is directed to the abstract operator algebraic point of view, here we analyze a class of particular transformations $\vpi \colon X \to X$ of a probability measure space $(X,\mu)$.

There are several well known examples (including the backward shift on infinite words on $N$ symbols, and finite Blaschke products with $N$ factors) where such transformations yield a representation $\al \colon L^\infty(X,\mu) \to L^\infty(X,\mu)$ that is implemented by a Cuntz family. That is, there is a Cuntz family $\{S_i\}_{i=1}^N$ in $\sB(L^2(X,\mu))$ such that
\begin{align*}
\al(M_f) = \sum_{i=1}^N S_i M_f S_i^*,
\end{align*}
where $M_f \in \sB(L^2(X,\mu))$ is the multiplication operator associated to $f \in L^\infty(X,\mu)$, and therefore $\al$ extends to a representation $\al_S$ of $\sB(L^2(X,\mu))$. Our goal here is two-fold. First we give conditions under which the transformation $\vpi \colon X \to X$ defines such a representation $\al \colon L^\infty(X,\mu) \to L^\infty(X,\mu)$ (Proposition \ref{P:decomposition}). Secondly we show that ergodicity of $\vpi \colon X \to X$ (as a transformation of a probability measure space) implies ergodicity of the induced $\al_S \colon \sB(L^2(X,\mu)) \to \sB(L^2(X,\mu))$ (as a representation of a von Neumann algebra) (Theorem \ref{T:ergodic}).

The existence of a Cuntz family implementing $\al \colon L^\infty(X,\mu) \to L^\infty(X,\mu)$ is connected to a decomposition of the space $X$ based on a maximal family of sets (Lemma \ref{L:decomp}). There is a question whether different decompositions yield the same extension. We show that the answer to this question is connected to the existence of an orthonormal basis of a suitable W*-module (Proposition \ref{P:pairing}). As a consequence we obtain a complete invariant on Stacey's multiplicity $n$ crossed products \cite{Sta93} (Corollary \ref{C:Stacey}).

A useful tool for the study of the endomorphism $\al \colon L^\infty(X,\mu) \to L^\infty(X,\mu)$ is the \emph{intertwining Hilbert module $\sE(X,\mu)$} introduced in section~4. Under certain conditions on the transformation $\vpi \colon X \to X$ there is a transfer operator; our setting encompasses several cases, including that described in \cite[Theorem 5.2]{CouMuhSch10}. We conclude by showing that the existence of a basis for the Hilbert module $\sE(X,\mu)$ is equivalent to the existence of a Cuntz family implementing $\al \colon L^\infty(X,\mu) \to L^\infty(X,\mu)$, and in turn is equivalent to the existence of a basis for $L^{\infty}(X, \mu)$ viewed as a Hilbert module, where the inner product is defined by the transfer operator (Theorem~\ref{T:Hilbertmodules}).

Hilbert modules may not have a well defined (up to unitary equivalence) basis, in contrast to Hilbert spaces.
Therefore it is central for (and non-trivial in) our analysis to achieve a well defined basis.
For example $\sO_2$ is unitarily equivalent to $\sum_{k=1}^n \sO_2$ for all $n \in \bbN$, as Hilbert modules over $\sO_2$ (Remark \ref{R:notuniquecard}).
This phenomenon is also connected to the multiplicity of multivariable C*-dynamics \cite{KakKat12} and produces a fascinating obstacle (so far) for the classification of these objects.
For encountering this problem, Gipson \cite{Gip14} develops the notion of the invariant basis number for C*-algebras, along with an in-depth analysis of C*-algebras that do (or do not) attain such a number.

\section{Preliminaries}

Let us begin with a general comment on $*$-endomorphisms $\al_S$ of $\sB(H)$ that are implemented by a Cuntz family $\{S_1, \dots, S_N\}$, i.e.,
\[
\al_S(T)=\sum_{i=1}^N S_iTS_i^*, \foral T \in \sB(H).
\]
We write $\sO_N = \ca(S_1, \dots, S_N)$ for the Cuntz algebra \cite{Cun77} inside $\sB(H)$. Both $\al_S$ and the restriction $\al_S|_{\sO_N}$ of $\al_S$ to $\sO_N$ are injective, but they are not onto for $N >1$. Indeed, if there is a $T \in \sB(H)$ such that $\al_S(T) =0$, then
\[
T = S_1^*S_1 T S_1^* S_1 = S_1^* \al_S(T) S_1 = 0.
\]
Furthermore if there is a $T \in \sO_N$ such that $\al_S(T) = S_1$, then
\[
I = S_1^*S_1 = S_1^* \al_S(T) = TS_1^*,
\]
hence $S_1$ is a unitary, which holds if and only if $N=1$.

Let $(X,\mu)$ and $(Y,\nu)$ be compact, Hausdorff measure spaces, endowed with their Borel structure. Then a continuous map $\vpi\colon (X,\mu) \to (Y,\nu)$ is a Borel homomorphism. However the mapping
\[
\al \colon (L^{\infty}(Y,\nu), \nor{\cdot}_\infty) \to (L^{\infty}(X,\mu), \nor{\cdot}_\infty): f \mapsto f\circ \vpi
\]
where $\nor{\cdot}_\infty$ is the essential sup-norm, may not even be well defined. In particular one can show that $\al$ is well-defined if and only if $\mu \circ \vpi^{-1} \ll \nu$ (i.e., $\vpi^{-1}$ \emph{preserves the $\nu$-null sets}). When $\vpi(Y)$ is in addition a Borel set, then $\al$ is well-defined and injective if and only if $\nu(\vpi(Y)^c)=0$ (i.e., $\vpi$ is \emph{almost onto $X$}) and $\mu \circ \vpi^{-1} \sim \nu$.

In general a Borel map $\vpi \colon X \to Y$ is said to preserve the $\nu$-null sets if $\nu\circ \vpi \ll \mu$. In this case $\nu \ll \mu \circ \vpi^{-1}$ and $\vpi(E)$ is Borel for every Borel subset $E$ of $X$. Indeed for the latter, observe that a Borel subset $E$ of $X$ is the union of an $F_\si$ set $A$ and a $\mu$-null set $N$. Then $\vpi(N)$ is a $\nu$-null set and compactness of $X$ implies that $A$ is $\si$-compact hence $\vpi(A)$ is Borel; thus $\vpi(E)$ is measurable.

Recall that if $\vpi\colon X \to Y$ is a Borel map, then a mapping $\psi\colon \vpi(X) \to X$ is called \emph{a Borel $($cross$)$ section of $\vpi$} if $\psi$ is a Borel map and $\vpi\circ \psi = \id_{\vpi(X)}$.

\begin{proposition}\label{P:isometry}
Let $\vpi\colon X \to Y$ be an onto map, such that $\vpi$ and $\vpi^{-1}$ preserve the null sets, and let $\psi \colon Y \to X$ be a Borel section of $\vpi$. Then $X_0:=\psi(Y)$ is Borel and there is an isometry $S \colon L^2(Y,\nu) \to L^2(X,\mu)$ such that
\[
M_{f\circ \vpi}|_{L^2(X_0,\nu|_{X_0})} = S M_f S^* \foral f\in L^\infty(Y,\nu).
\]
\end{proposition}

\begin{proof}
Observe that $\psi$ preserves the null sets (which implies that $X_0$ is Borel). Since $\mu \circ \vpi^{-1} \ll \nu$ then $\mu \ll \nu \circ \vpi$. For a null set $E \subseteq Y$ we get that $\nu \circ \vpi(\psi(E)) = \nu(E) = 0$, thus $\mu\circ \psi(E) = 0$. Note that since $X_0$ is Borel then $\vpi|_{X_0}$ is a Borel isomorphism with $\psi \colon Y \to X_0$ as an inverse.

On the other hand if $\mu \circ \psi(E) = 0$ then $\nu(E) = \nu \circ \vpi (\psi(E))=0$ since $\vpi$ preserves the null sets. Therefore $\nu$ is equivalent to $\mu \circ \psi = \mu|_{X_0} \circ \psi$, and the Radon-Nykodim derivative $u={d(\mu|_{X_0}\circ \psi)}/{d\nu}$ is defined. It is a standard fact that the operator $S_0^*\colon L^2(X_0, \mu|_{X_0}) \to L^2(Y,\nu)$ defined by
\[
S_0^*(g) = g \circ \psi \cdot u^{1/2}, \foral g \in L^2(X_0,\mu|_{X_0}),
\]
is a unitary such that
\[
M_{f\circ \vpi}|_{L^2(X_0,\nu|_{X_0})} = S_0 M_f S_0^* \foral f\in L^\infty(Y,\nu).
\]
Extend $S_0^*$ trivially to $S^*$ on $L^2(X,\mu) = L^2(X_0, \mu|_{X_0}) \oplus L^2(X^c, \mu|_{X_0^c})$. Then the adjoint $S$ of $S^*$ is an isometry and gives the required equation.
\end{proof}

\begin{proposition}\label{P:decomposition}
Let $\vpi\colon X \to Y$ be an onto map, such that $\vpi$ and $\vpi^{-1}$ preserve the null sets. Suppose that there is a family $\{\psi_1,\dots,\psi_N\}$ of $N$ Borel sections of $\vpi$ such that $\psi_i(Y) \cap \psi_j(Y) = \emptyset$ for $i \neq j$, and $\cup_i \psi(Y)$ is almost equal to $X$. Then there is a Cuntz family that implements $\al$.
\end{proposition}
\begin{proof}
For every $i=1,\dots,N$ let $X_i = \psi_i(Y)$ and let $S_i$ be constructed as above on $\sX_i=L^2(X_i,\mu|_{X_i})$. Note that $\sX_i \perp \sX_j$ for $i\neq j$, thus if $X_0 = \cup_i X_i$ and $\sX_0 = \oplus_i \sX_i$ we get $M_{\mathbf{1}|_{X_i}} =M_{\al(\mathbf{1})}|_{\sX_i} = S_i M_\mathbf{1} S_i^* = S_iS_i^*$, therefore
\[
M_{\al(f)}|_{\sX_0} = \sum_{i=1}^N M_{\al(f)}|_{\sX_i} = \sum_{i=1}^N S_i M_f S_i^*.
\]
Since $X= \cup_i X_i$ a.e. we obtain that $I_{\sX_0}= M_{\mathbf{1}}|_{\sX_0} = \sum_{i=1}^N M_{\mathbf{1}}|_{\sX_i} = \sum_{i=1}^N S_iS_i^*$. Finally $X_0$ is almost equal to $X$, hence $L^2(X,\nu) = \sX_0$ and the proof is complete.
\end{proof}

\section{Ergodic Extensions}

Let $(X,\mu)$ be a probability measure space such that $X$ is a compact, Hausdorff space and $\mu$ is a regular Borel measure on $X$. Then a measure preserving map $\vpi \colon X \to X$ induces an injective $*$-homomorphism $\al \colon L^\infty(X,\mu) \to L^\infty(X,\mu)$. We are interested in the case where $\al$ is implemented by a Cuntz family $\{S_i\}_{i=1}^N$ in $\sB(L^2(X,\mu))$. In this case $\al$ extends to an injective $*$-endomorphism $\al_S$ of $\sB(L^2(X,\mu))$. A natural question is whether ergodicity (of the mapping) $\vpi$ implies ergodicity (of the $*$-endomorphism) $\al_S$. Recall that $\al_S$ is \emph{ergodic} if the von Neumann algebra $\sN_{\al_S} := \{T \in \sB(\sH) \mid \al_S(T) = T\}$ is trivial. We aim to give a positive answer for a class of ergodic mappings that includes central examples.

Recall that a map $\vpi\colon X \to X$ is called a \emph{local homeomorphism} if for every point $x \in X$ there is a neighborhood $U$ such that $\vpi|_U$ is a homeomorphism onto its image. Clearly, local homeomorphisms are continuous and open. We begin with a decomposition lemma that fits in our study.

\begin{lemma} \label{L:decomp}
Let $\vpi$ be a local homeomorphism of a compact Hausdorff space $X$ such that $|\vpi^{-1}(x)| = N > 1$ for all $x \in X.$ Then there exist pairwise disjoint open subsets $U_1, \dots, U_N$ such that
\begin{enumerate}
\item \label{I:1} $\vpi|_{U_i}$ is one-to-one for all $i=1, \dots, n$;
\item \label{I:2} $\vpi(U_i) = \vpi(U_j)$ for all $i,j=1, \dots N$;
\item \label{I:3} $X= \cup_{i=1}^N (U_i \cup \partial U_i)$;
\item \label{I:4} $X= \vpi(U_i) \cup \partial \vpi(U_i)$ for all $i=1, \dots, N$.
\end{enumerate}
Moreover, $\vpi(\partial U_i) \subseteq \partial\vpi(U_i)$ and $\vpi^{-1}(\partial \vpi(U_i)) =
\cup_{j=1}^N \partial U_j$, for all $i=1, \dots, N$.
\end{lemma}
\begin{proof}
First let us construct a family that satisfies (1) and (2). Let $\sF$ be the collection that consists of $\{U_1,\dots, U_N\}$ such that $U_i$ are open, disjoint, $\vpi|_{U_i}$ is one-to-one for all $i=1,\dots,N$, and $\vpi(U_i)=\vpi(U_j)$.

\smallskip

\noindent \textit{Claim.} The collection $\sF$ is non-empty.

\noindent \textit{Proof of Claim.} Let a $y\in X$ and suppose that $x_1, \dots, x_N$ are the $N$ pre-images of $y$. Let $V_i$ be a neighborhood of $x_i$ such that $\vpi|_{V_i}$ is one-to-one. Since $X$ is a Hausdorff space we can choose $V_i$ be disjoint. Moreover $\vpi(U_i)$ are open sets, since $\vpi$ is an open map. Let $V= \cap_{i=1}^N \vpi(V_i)$ which is open and let $U_i = \vpi^{-1}(V) \cap V_i$. Then the $U_i$ are disjoint and $\vpi|_{U_i}$ is one-to-one, since the $U_i$ are subsets of the $V_i$. In addition
\begin{align*}
\vpi(U_i)
& =
\vpi\circ \vpi^{-1}(V) \cap \vpi(V_i)
 =
V \cap \vpi(V_i)
 =
V,
\end{align*}
and the proof of the claim is complete.

\smallskip

The collection $\sF$ is endowed with the partial order ``$\leq$'' such that
\[
\{U_1, \dots, U_N\} \leq \{V_1, \dots , V_N\} \, \text{ if } \, U_i \subseteq V_i, \foral i=1, \dots, N,
\]
after perhaps a re-ordering. Let $\fC= \left\{ \{U_1^k,\dots, U_N^k\} \mid k\in I \right\}$, be a chain in $\sF$, with the understanding that when $\{U_1^k,\dots, U_N^k\} \leq \{U_1^l,\dots, U_N^l\}$ then $U_i^k \subseteq U_i^{l}$ for all $i=1,\dots, N$. Then the element $\{ \cup_k U_1^k, \dots, \cup_k U_N^k\}$ is an upper bound for $\fC$ inside $\sF$. Indeed, what suffices to prove is that the $\cup_k U_i^k$ are disjoint (with respect to the indices $i$). If there were an $x$ in two such unions then there would be some $k,l \in I$ such that $x \in U_i^k \cap U_j^l$. Without loss of generality assume that $\{U_1^k,\dots, U_N^k\} \leq \{U_1^l,\dots, U_N^l\}$ therefore $x\in U_i^k \cap U_j^l \subseteq U_i^l \cap U_j^l=\emptyset$ which is absurd. Then the collection $\sF$ has a maximal element by Zorn's Lemma. From now on fix this maximal element be $\{U_1, \dots, U_N\}$. By definition the sets $U_1, \dots, U_N$ satisfy the properties (\ref{I:1}) and (\ref{I:2}) of the statement.

Secondly we prove that $X= \vpi(U_i) \cup \partial \vpi(U_i)$, for all $i=1, \dots, N$, where $\{U_1, \dots,U_N\}$ is the maximal family constructed above. Since $X \setminus \vpi(U_i)$ is closed it suffices to show that it has empty interior. To this end let $V$ be an open neighborhood of some $y \in \inte(X \setminus \vpi(U_i))$ with $N$ pre-images $x_1, \dots, x_N$. Then $\vpi^{-1}(V)$ is open, contains the $x_i$ and $\vpi^{-1}(V) \cap (\cup_{i=1}^N U_i) = \emptyset$. Indeed, if there was a $z\in \vpi^{-1}(V) \cap (\cup_{i=1}^N U_i)$, then $\vpi(z) = V \cap \vpi(U_i) = \emptyset$, which is absurd. As in the proof of the claim above, we can find neighborhoods $V_i$ of $x_i$ inside $\vpi^{-1}(V)$ such that they are disjoint, $\vpi|_{V_i}$ is one-to-one and $\vpi(V_i) = V$, perhaps by passing to a sub-neighborhood of $y$. Therefore the family $\{ U_1 \cup V_1, \dots, U_N \cup V_N \}$ is in $\sF$, which contradicts to the maximality of $\{U_1, \dots, U_N\}$.

Thirdly, we show that $X = \cup_{i=1}^N (U_i \cup \partial U_i)$. It suffices to show that the closed set $X \setminus \left(\cup_{i=1}^N U_i \right)$ has empty interior. Indeed, in this case it will coincide with its boundary, hence with $\partial \left(\cup_{i=1}^N U_i\right)$. Since the $U_i$ are open and disjoint we get that this boundary will be $\cup_{i=1}^N \partial(U_i)$. To this end, let $U$ be an open neighborhood of an element $x$ in the interior of $X \setminus \left(\cup_{i=1}^N (U_i \cup \partial U_i) \right)$. If there were an $x' \in U$ such that $\vpi(x') \in \vpi(U_i)$, then $\vpi(x')$ would have $N+1$ pre-images which is a contradiction. Indeed, recall that $\vpi(U_i) = \vpi(U_j)$ and $U_i \cap U_j = \emptyset$. Therefore $\vpi(U)$ is contained in the interior of $X \setminus \vpi(U_i)$. But $X \setminus \vpi(U_i)$ has empty interior, which gives the contradiction.

Finally, let $x\in \partial U_i$. If $\vpi(x) \in \vpi(U_i)$ then the element $\vpi(x)$ would have $N+1$ pre-images, which is a contradiction. Therefore $\vpi(\partial U_i) \subseteq X \setminus \vpi(U_i)= \partial \vpi(U_i)$. Note also that by construction we obtain
\begin{align*}
\vpi^{-1}(\partial \vpi(U_i))
=
\vpi^{-1}(X \setminus \vpi(U_i))
=
X \setminus \cup_{j=1}^N U_j
=
\cup_{j=1}^N U_j,
\end{align*}
for all $i=1, \dots, N$, and the proof of the lemma is complete.
\end{proof}

Let $\vpi$ be as in Lemma \ref{L:decomp} such that $\vpi$ and $\vpi^{-1}$ preserve the null sets. If $\{U_i\}_{i=1}^N$ is the family satisfying the properties of Lemma \ref{L:decomp} and
 \begin{align*}
\partial U_i \text{ (equivalently } \vpi(\partial U_i) \text {) are null sets}
\end{align*}
then the $*$-endomorphism $\al \colon L^\infty(X,\mu) \to L^\infty(X,\mu): f \to f \circ \vpi$ is implemented by a Cuntz family. Indeed, let $X_0 = \cup_{i=1}^N U_i$ and $Y_0=\vpi(U_i)$. Then $\vpi_0:= \vpi|_{X_0}$ has $N$ Borel sections $\psi_i$, for $i=1,\dots,N$, with
\[
\psi_i = [ \vpi|_{U_i} ]^{-1} \colon Y_0 \to X_0.
\]
Moreover $\vpi_0$ and $\vpi_0^{-1}$ preserve the null sets. By Proposition \ref{P:decomposition} there is a Cuntz family $\{S_i\}$ with
\[
S_i\colon L^2(Y_0,\mu|_{Y_0}) \to L^2(X_0,\mu|_{X_0})
\]
that implements the representation
\[
L^\infty(X_0,\mu|_{X_0}) \ni f \mapsto f\circ \vpi_0 \in L^\infty(Y_0,\mu|_{Y_0}).
\]
Since $X_0$ and $Y_0$ are almost equal to $X$ then the family $\{S_i\}_{i=1}^N$ implements $\al$.

Given a decomposition of $X$ as above and a finite word $\pmb{i}=i_1\dots i_k$ in $\{1,\dots,N\}$ we can define the Borel sets
\[
U_{i_1i_2\dots i_k} = \{ x\in X \mid x\in U_{i_1}, \dots, \vpi^{k-1} \in U_{i_k}\}.
\]
This is extended to infinite words $\pmb{i}=i_1i_2\dots i_k \dots$ with the understanding that $U_{\pmb{i}} = \cap_k U_{i_1\dots i_k}$.

\begin{theorem}\label{T:ergodic}
Let $(X, \mu, \vpi)$ be a dynamical system such that:
\begin{enumerate}
\item $\vpi$ is a local homeomorphism of $X$ such that each point of $X$ has $N > 1$ pre-images;
\item $\{U_i\}_{i=1}^N$ is a decomposition of $X$ as in Lemma~\ref{L:decomp} such that the $\partial U_i$ are null sets;
\item $\vpi$ is ergodic and preserves the null sets;
\item the sets $U_{\pmb{i}}$, for $\pmb{i} \in \bbF_N^+$ generate the $\si$-algebra up to sets of measure zero.
\end{enumerate}
Then $\al \colon M_f \mapsto M_{f \circ \vpi}$ admits an extension $\al_S$ to $\sB(L^2(X, m))$ which is ergodic. Furthermore $\al_S$ defines (by restriction) an irreducible representation of $\sO_N$.
\end{theorem}

\begin{proof}
Under these assumptions there is a Cuntz family $\{S_i\}_{i=1}^N$ that implements $\al$. Let $\al_S(T) = \sum_{i=1}^N S_i T S_i^*$ be the extension of $\al$ to $\sB(L^2(X, \mu))$. Since $\al_S$ is a weak$*$-continuous endomorphism of $\sB(L^2(X, \mu))$, then $\sN_{\al_S} = \{T \in \sB(\sH) \mid \al_S(T) = T\} $ is a von Neumann algebra. Fix a projection $P \in \sN_{\al_S}$. Then $\al_S(P) = P$ implies that $S_i P = P S_i$, and $S_i^*  P = P S_i^*$, for all $i=1,\dots,N$. In particular $P$ commutes with the range projections of the $S_i$ and the products of the $S_i$. But these projections are the characteristic functions of the sets $U_{\pmb{i}}$, for the words $\pmb{i}$ on the symbols $\{1, \dots, N\}$. Since the sets $U_{\pmb{i}}$ generate the $\si$-algebra up to null sets, the linear span of these projections is weak*-dense in $L^{\infty}(X, \mu)$. It follows that $P$ is in the MASA $L^{\infty}(X, \mu)$, hence $P = \chi_E$ for a measurable set $E$. However
\[
M_{\chi_E} = P = \al_S(P) = \al(P) = M_{\chi_E \circ \vpi}
\]
and ergodicity of $\vpi$ implies that $E$ is either $X$ or $\emptyset$. Thus $\sN_{\al_S} = \bbC I$. The second part of the theorem follows by the comments after \cite[Definition 3.2]{Lac93}.
\end{proof}

We give examples of dynamical systems which satisfy the conditions of Theorem \ref{T:ergodic}.

\begin{examples} \label{E:erg syst}
The first example is the canonical Cuntz-Krieger example of a dynamical system associated with Cuntz isometries. Let $N \in \bbN$ and
\[
X = \Pi_{k=1}^{\infty} \{1, \dots, N\}_k \text{ with measure } \mu = \Pi_{k=1}^{\infty} \mu_k
\]
where each $\mu_k = \mu_j$ for all $j, k$ such that $\mu_k(A) = |A|/N$ for all $A \subset \{1, \dots, N\}$. If we consider $X$ as a compact abelian group, with ``odometer'' addition, then $\mu$ is the Haar measure on $X$.
Let $\vpi$ be the shift map $\vpi(i_1, i_2, \dots) = (i_2, i_3, \dots)$ which is a $N$-to-one local homeomorphism. Then $\vpi$ is ergodic and the conditions of the theorem are satisfied for the cylinder sets $U_i:= \{(i_1, i_2, \dots ) \mid i_1 = i\}$ (which are clopen so that $\partial{U_i} = \emptyset$).

A second example arises when $X$ is the circle $\bbT$, $\mu$ is Lebesgue measure, and $\vpi$ is a finite Blaschke product with $N > 1$ factors and zero Denjoy-Wolf fixed point (i.e., at least one of the Blaschke factors is $z$). Then $\vpi$ is ergodic and the sets $U_i$ are arcs on the circle, so the condition $\mu(\partial{U_i}) = 0$ is satisfied. This example is considered in~\cite{CouMuhSch10}.
\end{examples}

In view of Theorem \ref{T:ergodic} one can ask whether the $\si$-algebra generated by the sets $U_{\pmb{i}}$ with $\pmb{i} \in \bbF_N^+$ always generates the full $\si$-algebra of measurable sets, up to measure zero. This is not true, as the following example shows.

\begin{example} \label{E:erg not enough}
Let $(X,\mu,\vpi)$ be the canonical Cuntz-Krieger example as above. Also let $\tau$ be an irrational rotation on the circle $\bbT$ with Lebesgue measure. Set $Y = X \times \bbT$, $\si(x, z) = (\vpi(x), \tau(z))$ and $\nu = \mu \times \la$. Then $(Y, \nu, \si)$ is ergodic as the product of the mixing shift map with the ergodic irrational rotation. If $U_i=\{(i_1, i_2, \dots) \in X \mid i_1 = i \}$ let $V_i =  U_i \times [0, 1]$. Then the $V_i$ are as in Lemma \ref{L:decomp}, but the $V_{\pmb{i}}$ with $\pmb{i} \in \bbF_N^+$ do not suffice to generate the $\si$-algebra of measurable sets up to measure zero.
\end{example}

\section{Uniqueness of the extension}

The reader is referred to the work of Paschke \cite{Pas73} for an introduction to W*-modules and to \cite{Lan95, ManTro05} for the general theory of C*-modules.

\begin{definition} \label{D:basis}
Let $\sM$ be a Hilbert module over a unital C*-algebra $A$. A subset $\{\xi_1, \dots \xi_N\}$ of $\sM$ is said to be an \emph{orthonormal basis for $\sM$} if $\xi_i \in \sM$, $\sca{\xi_i, \xi_j} = \de_{ij} 1_A$ and
\[
\xi = \sum_{i=1}^N  \xi_i \cdot \sca{\xi_i, \xi}, \foral \xi \in \sM.
\]
In the case where $N=\infty$ the sum is understood as norm-convergent.
\end{definition}

As a consequence $\sum_{i=1}^N \theta_{\xi_i,\xi_i} = \id_\sM$ with the understanding that the sum is convergent in the strong topology when $N=\infty$. When $A$ is non-unital, we define the basis of $\sM$ by using the unitization $A^1 = A + \bbC$. Indeed we can extend the right action to $A^1$ by
\[
\xi \cdot (a + \la) = \xi\cdot a + \la\xi,
\]
for all $a\in A$ and $\la \in \bbC$. Then the basis of $\sM$ over $A$ is defined as the basis of $\sM$ over $A^1$. This is just to ensure that the formula $\sca{\xi_i,\xi_j} = \de_{ij} 1_{A^1}$ makes sense.

\begin{remark} \label{R:notuniquecard}
In general, a Hilbert module may not have an orthonormal basis. However, W*-modules have a basis $\{\xi_i\}$ such that $\sca{\xi_i,\xi_i}$ is a projection \cite[Theorem 3.12]{Pas73}. Moreover, the size of an orthonormal basis is not well defined, meaning that there may be bases $\{s_i\}_{i\in I}$ and $\{t_j\}_{j\in J}$ with $|I| \neq |J|$. The reason is that the uniqueness of the linear combinations is not guaranteed. For a counterexample let $\sM=\sO_2$ be the trivial Hilbert module over itself, where $\sO_2$ is the Cuntz algebra on two generators, say $s_1$ and $s_2$. Then the sets $\{1_{\sO_2}\}$ and $\{s_1, s_2\}$ are both bases for the Hilbert module. Indeed for $\xi \in \sO_2$ we trivially have that $\xi = 1_{\sO_2} \cdot \sca{1_{\sO_2}, \xi}$, and that
\[
\xi = (s_1s_1^* + s_2s_2^*)\xi = s_1 \cdot \sca{s_1, \xi} + s_2 \cdot \sca{s_2, \xi},
\]
since $s_1s_1^* + s_2s_2^* = 1_{\sO_2}$.

Similarly, one can show that the trivial Hilbert module $\sM = \sO_2$ over $\sO_2$ is unitarily equivalent to the (interior) direct sum $\sM + \sM$ over $\sO_2$ by the unitary $U = \begin{bmatrix} s_1 & s_2 \end{bmatrix}$. Inductively we get that $\sM$ is unitarily equivalent to $\sum_{k=1}^n \sM$ for all $n \in \bbN$.
\end{remark}

\begin{remark} \label{R:uniquecard}
Nevertheless, when the Hilbert module is over a stably finite C*-algebra $A$ then the size is unique. Indeed, let $\{\xi_i\}_{i \in I}$ and $\{\eta_j\}_{j \in J}$ be two orthonormal bases of such a Hilbert module $\sM$ and form the rectangular matrix $U=[\sca{\xi_i,\eta_j}]$. Then, the $(i,j)$-entry of the $|I|\times |J|$ matrix $UU^*$ is
\begin{align*}
\sum_{k=1}^{|J|} \sca{\xi_i,\eta_k} \sca{\eta_k, \xi_j}
 =
\sum_k \sca{\xi_i, \eta_k \sca{\eta_k,\xi_j}}
 =
\sca{\xi_i,\xi_j}
 =
\de_{ij} 1_A.
\end{align*}
Analogous computations for $U^*U$ show that $U$ is a unitary in $M_{|I|,|J|}(A)$. Since $A$ is stably finite we get that $|I|=|J|$. In fact we get the following formula
\[
\left[\eta_1, \dots, \eta_N \right]
=
\left[\xi_1, \dots, \xi_N \right]
\left[U_{ij}\right],
\]
and the unitary $U$ is  in $M_N(A)$. In contrast to \cite{Lac93} the unitary $U$ may not be in $M_N(\bbC)$.
\end{remark}

Let $\al \colon L^\infty(X,\mu) \to L^\infty(X,\mu)$ be a *-homomorphism and let the linear space
\[
\sE(X,\mu) = \{ T \in \sB(L^2(X,\mu)) \mid  Ta = \al(a) T, \text{ for all } a \in L^\infty(X,\mu)\} .
\]
Then $\sE(X,\mu)$ becomes a Hilbert module over $L^\infty(X,\mu)$ by defining
\begin{align*}
& S \cdot a := Sa  \quad \text{and} \quad \sca{S,T}:=S^*T
\end{align*}
for all $a\in L^\infty(X,\mu)$ and $S,T \in \sE(X,\mu)$. Indeed, for $b\in L^\infty(X,\mu)$ we obtain
\[
(Sa)b = Sab = Sba = (Sb)a=(\al(b)S)a = \al(b) (Sa),
\]
thus $Sa \in \sE(X,\mu)$. Also,
\[
\sca{S,T} \cdot b = (S^*T)b = S^*Tb= S^*\al(b) T = bS^*T = b \cdot \sca{S,T},
\]
for all $b\in L^\infty(X,\mu)$, which implies that $\sca{S,T} \in L^\infty(X,\mu)' = L^\infty(X,\mu)$. Thus the inner product and the right action are well defined and routine calculations show that $\sE(X,\mu)$ is a Hilbert module over $L^\infty(X,\mu)$. In particular the Hilbert module $\sE(X,\mu)$ becomes a W*-correspondence over $L^\infty(X,\mu)$ by defining
\begin{align*}
& a\cdot S = \al(a)S, \foral a\in L^\infty(X,\mu) \text{ and } S \in \sE(X,\mu).
\end{align*}
Indeed, for $b\in L^\infty(X,\mu)$ we obtain that
\[
(a\cdot S)b = \al(a)Sb= \al(a)\al(b) S = \al(b)\al(a)S= \al(b)(a\cdot S),
\]
hence $a\cdot S \in \sE(X,\mu)$.

It is evident that $\sE(X,\mu)$ is a weak*-closed subspace of $\sB(L^2(X,\mu))$. Hence as a self-dual W*-correspondence it receives a basis $\{S_i\}_{i \in \sI}$ such that $\sca{S_i,S_j}= S_i^*S_j = 0$ when $i \neq j$, $\sca{S_i,S_i}= S_i^*S_i$ is a projection in $L^\infty(X,\mu)$, and $T = \sum_i S_i S_i^*T$, for all $T \in \sE(X,\mu)$ \cite[Theorem 3.12]{Pas73}.

\begin{lemma} \label{L:Ebasis}
Let $\{S_i\}_{i=1}^n$ be a basis for $\sE(X,\mu)$ with $N < \infty$. Then the following are equivalent:
\begin{enumerate}
\item $\{S_i\}_{i=1}^N$ is an orthonormal basis for $\sE(X,\mu)$;
\item $\{S_i\}_{i=1}^N$ is a Cuntz family that implements $\al$ of $L^\infty(X,\mu)$.
\end{enumerate}
\end{lemma}
\begin{proof}
For convenience we write $I \in \sB(L^2(X,\mu))$ also for the unit of $L^\infty(X,\mu)$. Since $\{S_i\}_{i=1}^N$ is a basis we obtain $I = \sum_{i=1}^N \theta_{S_i, S_i} = S_iS_i^*$. Moreover $S_i \in \sE(X,\mu)$ thus $S_i a = \al(a) S_i$ for all $a\in L^\infty(X,\mu)$. Hence
\[
\sum_{i=1}^N S_i a S_i^* = \al(a) \sum_{i=1}^N S_iS_i^* = \al(a).
\]

On the other hand if $\{ S_i\}_{i=1}^N $ is a Cuntz family, then $\sca{S_i, S_j} = \de_{ij}I $ and $\sum_{i=1}^N S_i S_i^* = I$, since $\al(I) = I$ for the unit $I \in L^\infty(X,\mu)$. If $\al(a)=\sum_{i=1}^N S_i a S_i^*$, then $ S_i a = \al(a) S_i $ for all $a \in L^\infty(X,\mu)$, for $i=1, \dots,N$.  Thus, $S_i \in \sE(X,\mu)$. For $T \in \sE(X,\mu)$ set $a_i = \sca{S_i, T} = S_i^*T$. Then
\[
\sum_{i=1}^N S_i a_i   = \sum_{i=1}^N S_i S_i^* T = T,
\]
and the proof is complete.
\end{proof}

Let $\{S_1, \dots, S_N\}$ be an orthonormal basis of $\sE(X,\mu)$, and let the extension $\al_S$ of $\al$ be given by
\[
\al_{S}\colon \sB(L^2(X,\mu)) \to \sB(L^2(X,\mu)): R \mapsto \sum_{i=1}^N S_i R S_i^*.
\]
We can then define the linear space
\[
\sH_{S}:=\{ T \in \sB(L^2(X,\mu)) \mid T R = \al_{S}(R) T \foral R \in \sB(L^2(X,\mu))\}.
\]
It becomes a Hilbert space endowed with the inner product
\begin{align*}
& \sca{T_1,T_2} = T_1^*T_2, \foral T_1, T_2 \in \sH_S.
\end{align*}
Indeed it is easy to check that $\sca{T_1, T_2} \in \sB(L^2(X,\mu))' = \bbC$. Moreover it has dimension $N$ and the Cuntz family $\{S_i\}_{i=1}^N$ is in $\sH_{S}$. The proof is the same as in Remark \ref{R:notuniquecard} taking into account that $\al_S(R)S_j = S_jR$, for all $R\in \sB(L^2(X,\mu))$. These results were established by Laca \cite{Lac93}.

\begin{proposition}\label{P:pairing}
Let $\{S_i\}_{i=1}^N$ and $\{Q_i\}_{i=1}^N$ be two orthonormal bases for $\sE(X,\mu)$. Then the following are equivalent:
\begin{enumerate}
\item The unitary $U$ that induces a pairing of the bases is in $M_N(\bbC)$;
\item The extensions $\al_{S}$ and $\al_{Q}$ in $\sB(L^2(X,\mu))$ coincide.
\end{enumerate}
\end{proposition}
\begin{proof}
For convenience we write $I \in \sB(L^2(X,\mu))$ also for the unit of $L^\infty(X,\mu)$.

\noindent $[(1) \Rightarrow (2)]$: We compute
\begin{align*}
\al_{Q}(R)
& =
\sum_{i=1}^N Q_i R Q_i^*
 =
\sum_{i,j,k=1}^N S_j\sca{S_j,Q_i} R \sca{Q_i,S_k} S_k^* \\
& =
\sum_{k,j=1}^N S_j R \sum_{i=1}^N \sca{S_j,Q_i}\sca{Q_i,S_k}S_k^*\\
& =
\sum_{k,j=1}^N S_j R \de_{j,k} S_k^*
 =
\sum_{k=1}^N S_k R S_k^*
 =
\al_{S}(R),
\end{align*}
since $\sum_{i=1}^N \sca{S_j,Q_i}\sca{Q_i,S_k}$ is the $(j,k)$-entry of $UU^*=I$, and we have used that the entry $\sca{S_j,Q_i}$ of $U$ is in $\bbC$.

\noindent $[(2) \Rightarrow (1)]$: If $\al_{Q} = \al_{S}$ then by definition $\sH_{S}= \sH_{Q}$. Thus
\[
S^*_iQ_k R = S^*_i \al_{Q}(R) Q_k = S^*_i \al_{S}(R) Q_k = R S^*_iQ_k
\]
for all $R\in \sB(L^2(X,\mu))$, hence $S^*_iQ_k \in \sB(L^2(X,\mu))'= \bbC$.
\end{proof}

As a consequence we have a complete invariant for the multiplicity $n$ crossed products on $L^\infty(X,\mu)$ \cite{Sta93}. Recall that given a $*$-endomorphism $\al \colon A \to A$ of a C*-algebra $A$ then \emph{the multiplicity $n$ crossed product $A \times_\al^n \bbN$} is the enveloping C*-algebra generated by $\pi(A)$ and a Toeplitz-Cuntz family $\{Q_i\}_{i=1}^n$ such that $\pi$ is a non-degenerate representation of $A$, and $\pi(\al(a)) = \sum_{i=1}^n Q_i\pi(a)Q_i^*$, for all $a\in A$. When $\al$ is unital then non-degeneracy of $\pi$ is redundant and $\{Q_i\}_{i=1}^n$ can be considered to be a Cuntz family \cite[Section 3 and Proposition 3.1]{KakPet13}. In \cite[Subsection 3.3]{KakPet13} we introduced the semicrossed product $A \times_\al \sT_n^+$ as the non-involutive subalgebra of $A \times_\al^n\bbN$ generated by $\pi(A)$ and $\{Q_i\}_{i=1}^n$.

\begin{corollary}\label{C:Stacey}
Let $\al$ be a unital weak*-continuous isometric endomorphism of $L^\infty(X,\mu)$ and suppose that there is a representation $(\id,\{S_i\}_{i=1}^n)$ of Stacey's crossed product $L^\infty(X,\mu) \times_\al^n \bbN$ on $L^2(X,\mu)$. Then the following are equivalent
\begin{enumerate}
\item $L^\infty(X,\mu) \times_\al^n \bbN \simeq L^\infty(X,\mu) \times_\al^m \bbN$ via a $*$-isomorphism that fixes $L^\infty(X,\mu)$ elementwise;
\item There is a representation $(\id,\{Q_i\}_{i=1}^m)$ of $L^\infty(X,\mu) \times_\al^m \bbN$ acting on $L^2(X,\mu)$;
\item $n=m$;
\item $L^\infty(X,\mu) \times_\al \sT_n^+ \simeq L^\infty(X,\mu) \times_\al \sT_m^+$ via a completely isometric isomorphism that fixes $L^\infty(X,\mu)$ elementwise.
\end{enumerate}
\end{corollary}
\begin{proof}
The fact that $\al$ is an isometric endomorphism of a C*-algebra implies that it is a $*$-homomorphism of the C*-algebra $L^\infty(X,\mu)$ and the multiplicity $n$ crossed products are well defined. The implication $[(3) \Rightarrow (4)]$ is immediate.

\noindent $[(4)\Rightarrow (1)]$: By \cite[Theorem 3.13]{KakPet13} the C*-algebra $L^\infty(X,\mu) \times_\al^n \bbN$ is the C*-envelope of $L^\infty(X,\mu) \times_\al \sT_n^+$, thus the completely isometric isomorphism extends to a $*$-isomorphism of the corresponding C*-algebras.

\noindent $[(1) \Rightarrow (2)]$: If $\Phi$ is the $*$-isomorphism, let $Q_i := \Phi(S_i)$.

\noindent $[(2) \Rightarrow (3)]$: Let $(\id, \{S_i\}_{i=1}^n)$ and $(\id, \{Q_i\}_{i=1}^m)$ be two such representations. Then $\al$ is implemented by $\{S_i\}_{i=1}^n$ and $\{Q_i\}_{i=1}^m$, thus they define a basis for $\sE(X,\mu)$. Therefore $n=m$ by Remark \ref{R:uniquecard}.
\end{proof}

\section{Existence of a transfer operator}

In general the mapping $C(X) \ni f \stackrel{C_\vpi}{\longmapsto} f\circ \vpi \in L^2(X,\mu)$ may not extend to an operator on the Hilbert space $L^2(X,\mu)$. However if
\[
c_0 \nor{\xi}_2 \leq \nor{C_\vpi \xi}_2 \leq c_1 \nor{\xi}_2, \foral \xi \in L^2(X,\mu),
\]
then $C_\vpi$ is an injective operator in $\sB(L^2(X,\mu))$, and $\vpi^{-1}$ preserves the null sets. The map $\mu$ is called \emph{$\vpi$-bounded} if there is a constant $K > 0$ such that $\mu(\vpi(E)) \leq K \mu(E)$, for all measurable sets $E \subset X$. In this case $\vpi$ preserves also the $\mu$-null sets.

Under these assumptions let the polar decomposition $C_\vpi =S_\vpi a_\vpi$. Then $S_\vpi$ is an isometry and $a_\vpi$ is invertible. We can check that by definition $C_\vpi a = \al(a) C_\vpi$ for all $a \in L^\infty(X,\mu)$, hence $C_\vpi \in \sE(X,\mu)$. Hence $a_{\vpi}^2 = C_\vpi^* C_\vpi \in L^\infty(X,\mu)'$ so $a_{\vpi} \in L^\infty(X,\mu)$. Consequently, the isometry $S_\vpi = C_\vpi a_{\vpi}^{-1}$ is also in $\sE(X,\mu)$ and the mapping
\[
\sL\colon L^\infty(X,\mu) \rightarrow L^\infty(X,\mu): a \mapsto S_\vpi^* a S_\vpi,
\]
defines a \emph{transfer operator of $\al$}, i.e. $\sL$ is positive and $a\sL(b) = \sL(\al(a)b)$ for all $a,b \in L^\infty(X,\mu)$. Following Exel \cite{Exe03} let the semi-inner-product on the $L^\infty(X,\mu)$-module $L^\infty(X,\mu)_\sL$ given by
\begin{align*}
\sca{\eta,\xi}_\sL = \sL(\eta^*\xi), \text{ and } \xi \cdot a = \xi \al(a),
\end{align*}
for all $\eta,\xi, a \in L^\infty(X,\mu)$.

\begin{proposition}
Assume that $\mu$ is $\vpi$-bounded and that $C_\vpi$ is a bounded below operator of $\sB(L^2(X,\mu))$. Then $L^\infty(X,\mu)_\sL$ is a Hilbert module over $L^\infty(X,\mu)$, and as a vector space it coincides with $L^\infty(X,\mu)$.
\end{proposition}
\begin{proof}
It suffices to show that the norm $\nor{\cdot}_\sL$ on the module $L^\infty(X,\mu)_\sL$ is equivalent to the norm $\nor{\cdot}$ of $L^\infty(X,\mu)$.

First we show that there is a constant $M$ such that $\nor{a} \leq M \nor{aS_\vpi}$ for every $a\in L^\infty(X,\mu)$. Since $\nor{a S_\vpi}^2 = \nor{|a| S_\vpi}^2$ and $\nor{a}=\nor{|a|}$, it is enough to show that the relation  $\nor{a} \leq M \nor{aS_\vpi}$ holds for all positive $a$ in the norm-dense subspace of simple functions. To this end let $a = \sum_{i=1}^n d_i \chi_{E_i}$ where the sets $E_i$ are disjoint, of positive measure, and $d_1 > d_2 > \cdots > d_n > 0$; hence $\nor{a} = d_1$. To compute the norm $\nor{aS_\vpi}$ we let $a$ act on unit vectors in the range of $C_\vpi$; equivalently with unit vectors in the range of $S_\vpi$.
Let $E = E_1$ and $\xi = \frac{1}{\sqrt{\mu(\vpi^{-1}(\vpi(E)))}} \chi_{\vpi^{-1}(\vpi(E))}$. Then $\xi $ is a unit vector in the range of $S_\vpi$. Also, the assumptions on $S_\vpi$ and $\mu$ imply that $\mu(\vpi^{-1}(\vpi(E))) \leq c_1^2 \mu(\vpi(E)) \leq c_1^2 K \mu(E)$. Therefore
\begin{align*}
\nor{aS_\vpi}^2
& \geq
\nor{a \xi}_2^2
 =
\int_X a^2 |\xi|^2 \, d\mu
 \geq
\int_X d_1^2 \chi_E |\xi|^2 \, d\mu\\
& =
\frac{1}{\mu(\vpi^{-1}(\vpi(E)))} \int_X d_1^2 \chi_{E} \, d\mu
 =
\frac{\mu(E)}{m(\vpi^{-1}(\vpi(E)))} d_1^2
 \geq
\frac{1}{c_1^2K} d_1^2.
\end{align*}
Since $\nor{a} = d_1,$ we have that $\nor{aS_\vpi} \geq \frac{1}{c_1\sqrt{K}} \nor{a} $ on a norm dense subspace.

By the above inequality we obtain the equivalence of the norms $\nor{\cdot}_\sL$ and $\nor{\cdot}$. Indeed we have that
\begin{align*}
\frac{1}{M^2}\nor{a}^2 
& \leq 
\nor{|a|S_\vpi}^2
= 
\nor{S_\vpi^* |a|^2 S_\vpi} 
= 
\nor{\sL(a^*a)} \\
& = 
\nor{a}_\sL^2 
 =
\nor{S_\vpi^* a^*a S_\vpi} 
\leq 
\nor{a^*a} = \nor{a}^2
\end{align*}
where we have used that $S_\vpi$ is an isometry, and the proof is complete.
\end{proof}

The following theorem is the analogue of \cite[Theorem 5.2]{CouMuhSch10}.

\begin{theorem} \label{T:Hilbertmodules}
Assume that $\mu$ is $\vpi$-bounded and that $C_\vpi$ is a bounded below operator of $\sB(L^2(X,\mu))$. Then the following are equivalent:
\begin{enumerate}
\item $\{\xi_i\}_{i=1}^n$ is an orthonormal basis for the Hilbert module $L^\infty(X,\mu)_\sL$;
\item $\{\xi_i S_\vpi\}_{i=1}^n$ is an orthonormal basis for the Hilbert module $\sE(X,\mu)$;
\item $\{\xi_i S_\vpi\}_{i=1}^n$ is a Cuntz family that implements $\al$.
\end{enumerate}
\end{theorem}
\begin{proof}
It will be convenient to denote $S_i:=\xi_i S_\vpi$. Note that by definition $S_i \in \sE(X,\mu)$ and recall that the equivalence $[(2) \Leftrightarrow (3)]$ is Lemma \ref{L:Ebasis}. Moreover we write $I \in \sB(L^2(X,\mu))$ also for the unit of $L^\infty(X,\mu)$. The constant function of $L^2(X,mu)$ will be denoted by $\mathbf{1}$.

\noindent $[(1) \Rightarrow (3)]$: First we have that the $S_i$ have orthogonal ranges, since
\[
S_i^* S_j = S_\vpi^* \xi_i^* \xi_j S_\vpi
= \sL(\xi_i^* \xi_j ) = \sca{\xi_i, \xi_j}_\sL = \de_{ij} I.
\]
Recall that the constant function $\mathbf{1} \colon X \to \bbC$ is a separating vector and $C_\vpi(\mathbf{1})=\mathbf{1}\circ \vpi=\mathbf{1}$. Therefore,
\begin{align*}
S_i a S_i^*(\mathbf{1})
& =
S_i a S_i^* C_\vpi(\mathbf{1})
 =
\xi_i S_\vpi a S_\vpi^* \xi_i^* S_\vpi a_\vpi (\mathbf{1})\\
& =
\al(a) \xi_i S_\vpi \sL(\xi_i^*\al(a_\vpi)) (\mathbf{1})
 =
\al(a) \xi_i \al( \sL(\xi^*_i\al(a_\vpi))) S_\vpi (\mathbf{1}).
\end{align*}
for all $i = 1, \dots, n$. Since $\{\xi_i\}_{i=1}^n$ defines a basis of $L^\infty(X,\mu)_\sL$ we have that $a= \sum_{i=1}^n \xi_i \cdot \sca{\xi_i,a}_\sL = \sum_{i=1}^n \xi_i \al(\sL(\xi_i^*a))$ for all $a\in L^\infty(X,\mu)$. Thus
\begin{align*}
\sum_{i=1}^n S_i a S_i^* (\mathbf{1})
& =
\sum_{i=1}^n \al(a) \xi_i \al(\sL(\xi^*_i\al(a_\vpi))) S_\vpi (\mathbf{1}) \\
& =
\al(a) \sum_{i=1}^n \xi_i \al(\sL(\xi^*_i\al(a_\vpi))) S_\vpi (\mathbf{1}) \\
& =
\al(a) \al(a_\vpi) S_\vpi (\mathbf{1})
 =
\al(a) S_\vpi a_\vpi (\mathbf{1}) \\
& =
\al(a) C_\vpi(\mathbf{1})
 =
\al(a)(\mathbf{1}).
\end{align*}
Since $\mathbf{1}$ is a separating vector we obtain that $\{S_i\}$ implements $\al$.

\noindent $[(3) \Rightarrow (1)]$: Note that the functions $\xi_i$ are orthonormal, since
\[
\sca{\xi_i, \xi_j}_\sL = \sL(\xi_i^* \xi_j) = C_\vpi^* \xi_i^* \xi_j C_\vpi = S_i^* S_j = \de_{ij}I .
\]
To see that the $\{\xi_i\}_{i=1}^n$ span $L^\infty(X,\mu)_\sL$, let an element $a \in L^\infty(X,\mu)$ with $\sca{\xi_i, a}_\sL = 0$ for all $i$. Then
\begin{align*}
(a S_\vpi)^*
& =
S_\vpi^* a^* \cdot \sum_{i=1}^n S_i S_i^*
 =
S_\vpi^* a^* \cdot \sum_{i=1}^n \xi_i S_\vpi S_i^* \\
& =
\sum_{i=1}^n (S_\vpi^* a^* \xi_i S_\vpi) S_i^*
 =
\sum_{i=1}^n \sca{a, \xi_i}_\sL S_i^*
= 0,
\end{align*}
so that $a S_\vpi = 0$. Hence $a C_\vpi=0$, thus $a(\mathbf{1}) = a C_\vpi(\mathbf{1}) = 0$. Since $\mathbf{1}$ is a separating vector we obtain that $a=0$.
\end{proof}

\begin{remark} \label{R:basis}
Assume that $\vpi\colon X \rightarrow X$ has $N$ Borel sections as in Proposition \ref{P:decomposition}. Then the $N$ isometries $S_i$ of Proposition \ref{P:decomposition} can be written as
\begin{align*}
S_i
  =
M_{\chi_{Y_i}} M_{u_i} C_\vpi a_{\vpi}
 =
M_{\chi_{Y_i}} M_{u_i} \al(a_\vpi) C_\vpi
 =
M_{\chi_{Y_i}} M_{u_i} M_{h\circ \vpi} C_\vpi,
\end{align*}
where $u_i $ are as in Proposition \ref{P:isometry} for $\psi = \psi_i$ and $a_\vpi=M_h \in L^\infty(X,\mu)$. Therefore, the elements $\xi_i =  M_{\chi_{Y_i}} M_{u_i} M_{h\circ \vpi} \in L^\infty(X,\mu)$ define a basis for $L^\infty(X,\mu)_\sL$.

There is also a converse of the above scheme that works at the level of $*$-homomorphisms. We would like to thank Philip Gipson for bringing this to our attention. If there is a Cuntz family $\{S_i\}_{i=1}^n$ in $\sB(L^2(X,\mu))$ that implements $\al$ then $S_i^* a S_i \in L^\infty(X,\mu)$ for all $i=1, \dots, n$. This follows because $L^\infty(X,\mu)$  is a MASA, $S_i b = \al(b) S_i$, and
\[
S_i^* a S_i \cdot b = S_i^* \al(b) a S_i = b \cdot S_i^* a S_i,
\]
for all $b \in L^\infty(X,\mu)$. Furthermore $S_iS_i^*$ commutes with every $a \in L^\infty(X,\mu)$, thus the $*$-homomorphisms $\be_i \colon L^\infty(X,\mu) \to L^\infty(X,\mu)$ given by $\be_i(a)=S_i^*aS_i$ are $n$ left inverses for $\al$.
\end{remark}


\end{document}